\documentclass[11pt]{amsart}

\usepackage[latin1]{inputenc}
\usepackage[T1]{fontenc}
\usepackage{amsmath,amsfonts,amssymb,amsthm,url,geometry,mathrsfs,color,esint,bm,mathtools,tikz}

\newcommand{\st}{\ : \ }

\renewcommand{\leq}{\leqslant}
\renewcommand{\geq}{\geqslant}

\newcommand{\R}{\mathbf{R}}
\newcommand{\C}{\mathbb{C}}

\newcommand{\M}{\mathsf{M}}

\DeclareMathOperator{\card}{\mathrm{card}}

\DeclareMathOperator{\mathspan}{\mathrm{span}}
\DeclareMathOperator{\aff}{\mathrm{aff}}

\DeclareMathOperator{\co}{\mathcal{C}}

\newcommand{\scalar}[2]{\langle #1 , #2\rangle}
\newcommand{\ketbra}[2]{| #1 \rangle \langle  #2 |}
\newcommand{\ket}[1]{| #1 \rangle}

\theoremstyle{plain}
\newtheorem{theorem}{Theorem}

\newtheorem{proposition}{Proposition}

\newtheorem{lemma}[theorem]{Lemma}

\theoremstyle{remark}

\usepackage{graphicx} 
\graphicspath{{figures/} }
\usepackage{dsfont}

\title{Maximal exponent of the Lorentz cones}
\author{Guillaume Aubrun}
\address{\small{Institut Camille Jordan, Universit\'{e} Claude Bernard Lyon 1, 43 boulevard du 11 novembre 1918, 69622 Villeurbanne cedex, France}}
\email{aubrun@math.univ-lyon1.fr}

\author{Jing Bai}
\address{\small{School of Mathematics, Harbin Institute of Technology, 92 West Dazhi Street, Nangang District, 150001 Harbin, China
\newline{and Institut Camille Jordan, Universit\'{e} Claude Bernard Lyon 1, 43 boulevard du 11 novembre 1918, 69622 Villeurbanne cedex, France
}}}
\email{jingb@stu.hit.edu.cn}


\begin{document}

\begin{abstract} 
We show that the maximal exponent (i.e., the minimum number of iterations required for a primitive map to become strictly positive) of the $n$-dimensional Lorentz cone is equal to $n$. As a byproduct, we show that the optimal exponent in the quantum Wielandt inequality for qubit channels is equal to $3$.
\end{abstract}

\maketitle

Our main object of study is the $n$-dimensional \emph{Lorentz cone} (also known as second-order cone, quadratic cone, or ice-cream cone), which is the cone $\mathcal{L}_{n} \subset \R^n$ defined as
\[ \mathcal{L}_{n} = \left\{ (x_1,\dots,x_n) \in \R^n \st x_n \geq \left( x_1^2 + \dots + x_{n-1}^2 \right)^{1/2} \right\} .\]
We denote by $\mathrm{int}(\mathcal{L}_n)$ the interior of $\mathcal{L}_n$. We say that a linear map $\Psi : \R^n \to \R^n$ is \emph{positive} if $\Psi(\mathcal{L}_n) \subset \mathcal{L}_n$, \emph{strictly positive} if $\Psi(\mathcal{L}_n \setminus \{0\}) \subset \mathrm{int}(\mathcal{L}_n)$ and \emph{primitive} if it is positive and if there exists an integer $k \geq 1$ such that $\Psi^k$ is strictly positive. If $\Psi$ is primitive, the smallest such $k$ is called the \emph{primitivity index} of $\Psi$ and denoted~$\gamma(\Psi)$.
The main result of this paper is the following theorem.

\begin{theorem} \label{theorem:main}
Let $n \geq 1$. If $\Psi : \R^n \to \R^n$ is primitive, then $\gamma(\Psi) \leq n$. Moreover, there is a primitive map $\Psi : \R^n \to \R^n$ such that $\gamma(\Psi)=n$.
\end{theorem}

As we explain later, this theorem can be seen as the affine or projective analogue of the following classical result by Pt\'ak \cite{Ptak62}: if $A$ is an $n \times n$ matrix, then $\rho(A) = \|A\|$ if and only if $\|A^n\|=\|A\|^n$ (we denote by $\rho(A)$ and $\|A\|$ respectively the spectral radius and operator norm of $A$).

The paper is organized as follows. Section \ref{section:background} contains background and connects to related works, as well as a reformulation of Theorem \ref{theorem:main} involving affine self-maps of the Euclidean ball. The bound $\gamma(\Psi) \leq n$ is proved in Section \ref{section:upperbound} and the sharpness of this inequality follows from the example constructed in Section \ref{section:lowerbound}. Finally, when specialized to $n=4$, our result has an implication in quantum information theory which we develop in Section \ref{section:qubits}.

\section{Introduction} \label{section:background}

\subsection{Cones, maximal exponent}

We work in a finite-dimensional real vector space~$V$. A subset $\co \subset V$ is said to be a \emph{convex cone} if, for every $x,y \in \co$ and $s,t \in \R_+$, we have $sx+ty \in \co$. A cone $\co$ is said to be \emph{proper} if it is closed, salient (i.e., $\co \cap (-\co) = \{0\}$) and generating (i.e., $\co - \co =V$). 

We extend to the setting of a proper cone $\co \subset V$ the concepts of positivity and primitivity defined earlier for the Lorentz cones. A linear map $\Psi : V \to V$ is \emph{$\co$-positive} if $\Psi(\co) \subset \co$, \emph{strictly $\co$-positive} if $\Psi(\co \setminus \{0\}) \subset \mathrm{int}(\co)$ and   \emph{$\co$-primitive} if it is $\co$-positive and if there exists an integer $k \geq 1$ such that $\Psi^k$ is strictly $\co$-positive. If $\Psi$ is $\co$-primitive, the smallest such~$k$ is called the \emph{$\co$-primitivity index} of $\Psi$ and denoted $\gamma(\co,\Psi)$.

The \emph{maximal exponent} of $\co$, denoted $\gamma(\co)$, is the supremum of $\gamma(\co,\Psi)$ over all $\co$-primitive maps $\Psi : V \to V$. With this notation, the statement of Theorem \ref{theorem:main} reads as the equality $\gamma(\mathcal{L}_n)=n$.

We use the finite-dimensional version of the Krein--Rutman theorem (see \cite[Theorem~19.2]{Deimling85}): every $\co$-positive operator $\Psi$ has an eigenvector $x \in \co$ associated to the eigenvalue~$\rho(\Psi)$ (the spectral radius of $\Psi$). If moreover $\Psi$ is $\co$-primitive, then necessarily $\rho(\Psi)>0$ (otherwise $\Psi$ would be nilpotent, contradicting $\co$-primitivity) and $x \in \mathrm{int}(\co)$.

\subsection{Duality}

If $\co \subset V$ is a cone, its \emph{dual cone} is the cone in the dual vector space $V^*$ defined as
\[ \co^* = \{ f \in V^* \st \scalar{f}{x} \geq 0 \textnormal{ for every } x \in \co \}. \]
If $\co$ is proper, then $\co^*$ is also proper. The bipolar theorem asserts that $(\co^*)^*=\co$ provided we identify the double dual space $V^{**}$ with $V$. The Lorentz cone $\mathcal{L}_n$ is self-dual: if we identify the vector space $\R^n$ with its own dual using the standard inner product, then $\mathcal{L}_n^*=\mathcal{L}_n$.

A \emph{sole} of a proper cone $\co$ is a set of the form $\{x \in \co \st f(x)=\alpha \}$, where $f \in \mathrm{int} (\co^*)$ and $\alpha > 0$. If $K$ is a sole of $\co$, then $K$ is compact and $\co = \{ \lambda x \st x \in K, \lambda \geq 0\}$. 
 
We have the relation
\[ \mathrm{int} (\co) = \{ x \in V \st \scalar{f}{x} > 0 \textnormal{ for every } f \in \co^* \setminus \{0\} \}. \]

Given a linear map $\Psi : V \to V$
\[ \Psi \textnormal{ is }\co\textnormal{-positive} \iff \scalar{f}{\Psi(x)} \geq 0 \textnormal{ for every } x \in \co, f \in \co^* \]
\[ \Psi \textnormal{ is strictly}\co\textnormal{-positive} \iff \scalar{f}{\Psi(x)} > 0 \textnormal{ for every } x \in \co \setminus \{0\}, f \in \co^* \setminus \{0\}.\]
It is clear from these formulas that $\Psi$ is $\co$-positive (resp.\ strictly $\co$-positive, rep $\co$-primitive) if and only if the adjoint map $\Psi^* : V^* \to V^*$ is $\co^*$-positive (resp.\ strictly $\co^*$-positive, resp.\ $\co^*$-primitive). Moreover the cones $\co$ and $\co^*$ have the same maximal exponents.

\subsection{Affine maximal exponent}

Let $K$ be a convex body (i.e., a compact convex set of full dimension) in a finite-dimensional affine space $W$. An affine map $\Phi : W \to W$ is said to be
\emph{$K$-positive} if $\Phi(K) \subset K$, \emph{strictly $K$-positive} if $\Phi(K) \subset \mathrm{int}(K)$ and \emph{$K$-primitive} if it is $K$-positive and if there exists a positive integer $k$ such that $\Phi^k$ is strictly $K$-positive. If~$\Phi$ is $K$-primitive, the smallest integer $k$ with this property is called the \emph{affine $K$-primitivity index} of $\Phi$ and denoted $\gamma_{\mathrm{aff}}(K,\Phi)$. The \emph{affine maximal exponent} of $K$, denoted $\gamma_{\mathrm{aff}}(K)$, is the supremum of $\gamma_{\mathrm{aff}}(\Phi)$ over all $K$-primitive affine maps $\Phi : W \to W$.

If $X$ is a finite-dimensional normed space with unit ball $B$, observe that a linear map $T : X \to X$ is $B$-positive (resp.\ strictly $B$-positive) if and only if it has operator norm $\leq 1$ (resp.~$< 1$). Moreover,~$T$ is $B$-primitive if and only if it has operator norm $\leq 1$ and spectral radius $<1$. The supremum of $\gamma(B,T)$ over $B$-primitive linear maps $T$ has been studied in the Banach space literature as the \emph{critical exponent} of the normed space $X$. We refer to~\cite{Ptak93} for a survey on critical exponents.

The next proposition states that the maximal exponent of a cone is the supremum of affine maximal exponents of its soles. While this statement is folklore, we could note locate it in the literature and include a proof.

\begin{proposition} \label{prop:sup-over-soles}
Let $\co \subset V$ be a proper cone. Then 
\begin{equation} \label{eq:sup-over-soles} \gamma(\co) = \sup_{K \textnormal{ sole of }\co} \gamma_{\aff}(K).
\end{equation}
\end{proposition}

\begin{proof}
Given $f \in \mathrm{int} (\co^*)$ and $\alpha > 0$, consider the affine hyperplane $W = \{ x \in V \st f(x)=\alpha \}$ and the sole of $\co$ given by $K= \co \cap W$. Any affine map $\Phi : W \to W$ can be extended uniquely into a linear map $\Psi : V \to V$. Moreover, the affine map $\Phi$ is $K$-positive (resp., strictly $K$-positive, \mbox{$K$-primitive}) if and only if the linear map $\Psi$ is $\co$-positive (resp., strictly $\co$-positive, $\co$-primitive). We have therefore $\gamma_{\aff}(K,\Phi) = \gamma(\co,\Psi)$ and the inequality $\geq$ in equation 
\eqref{eq:sup-over-soles} follows by taking supremum over $K$ and~$\Phi$.

Conversely, let $\Psi : V\to V$ be a $\co$-primitive map. Its spectral radius $\rho(\Psi)$ is nonzero and we may assume  by rescaling that $\rho(\Psi)=1$. By the Krein--Rutman theorem, the adjoint map~$\Psi^*$, which is $\co^*$-primitive, admits an eigenvector $f \in \mathrm{int}(\co^*)$ for the eigenvalue $1$. Consider the affine hyperplane $W = \{ x \in V \st f(x)=1 \}$ and the sole $K = \co \cap W$. Since $\Psi(W)\subset W$, the linear map $\Psi$ induces by restriction a $K$-primitive affine map $\Phi : W \to W$. As before, we have $\gamma_{\aff}(K,\Phi) = \gamma(\co,\Psi)$ and the inequality $\leq$ in equation \eqref{eq:sup-over-soles} follows by taking supremum over~$\Psi$.
\end{proof}

We denote by $B_n$ the unit ball of the standard Euclidean space $\R^n$. Any sole of the Lorentz cone $\mathcal{L}_{n+1}$ is affinely isomorphic to $B_{n}$. By Proposition \ref{prop:sup-over-soles}, Theorem \ref{theorem:main} can be equivalently stated as follows.

\begin{theorem} \label{theorem:affine}
For every integer $n \geq 1$, we have $\gamma_{\mathrm{aff}}(B_n)=n+1$.
\end{theorem}

Sections \ref{section:upperbound} and \ref{section:lowerbound} are devoted to the proof of Theorem \ref{theorem:affine}: in Section \ref{section:upperbound} we prove that any \mbox{$B_n$-primitive} affine map $\Psi : \R^n \to \R^n$ satisfies $\gamma_{\aff}(B_n,\Psi) \leq n+1$, and in Section \ref{section:lowerbound} we construct an example showing that this inequality is sharp. 

\subsection{Related works}

The study of maximal exponents of cones can be traced back to the classical result by Wielandt~\cite{wie} which asserts that the maximal exponent of the cone $\R_+^n$ equals $(n-1)^2+1$ (Wielandt's original proof was only published posthumously in \cite{Schneider02}). The maximal exponents of polyhedral cones have been studied in detail in the series of papers \cite{LT2010a,LT2010b,LPT2013}. We also mention that there exist cones for which the maximal exponent is infinite (see \cite[Section 6]{LT2010a}).

Our result is closely related to Pt\'ak's theorem \cite{Ptak62} stating that the critical exponent of the $n$-dimensional Euclidean space $\ell_2^n$ equals $n$. This means that if $\Phi$ is a linear contraction on~$\ell_2^n$ with spectral radius $<1$, then its $n$th iteration $\Phi^n$ maps the unit ball into its interior. Our Theorem \ref{theorem:affine} shows that for affine maps, one more iteration is necessary and sufficient to achieve this property.

We develop in Section \ref{section:qubits} the case of the cone of positive semidefinite complex matrices, which is relevant in quantum information theory in the context of the quantum Wielandt inequality.

\section{Upper bound on the maximal exponent} \label{section:upperbound}

Throughout this section and the following one, we fix an integer $n \geq 1$ and we use the terminology  ``positive'', ``strictly positive'' and ``primitive'' to mean ``$B_n$-positive'', ``strictly $B_n$-positive'' and ``$B_n$-primitive''. We denote by $S^{n-1} = \partial B_n$ the unit sphere in the Euclidean space $\R^n$. Given a subset $X \subset \R^n$, we denote by $\aff (X)$ the affine subspace generated by $X$. We start with a simple lemma.

\begin{lemma}\label{nonconstant}
If an affine map $\Phi : \R^n \rightarrow \R^n$ is positive and nonconstant, then $\Phi(\mathrm{int}(B_n)) \subset \mathrm{int}(B_n)$.
\end{lemma}

\begin{proof}
Take $x \in \mathrm{int} (B_n)$ and assume by contradiction that $\Phi(x) \in S^{n-1}$. Let $V$ be a open ball centered at $x$ and contained in $B_n$. For every $y \in V$, the point $z = 2x-y$ is in $V$ and we have $x = \frac{y+z}{2}$, hence $\Phi(x)= \frac{\Phi(y)+\Phi(z)}{2}$. Since $\Phi(x)$ is an  extreme point of $B_n$, it follows that $\Phi(y)=\Phi(z)=\Phi(x)$. The affine function $\Phi$ is constant on $V$ hence constant on $\R^n = \mathrm{aff} (V)$, leading to a contradiction.
\end{proof}

Given a positive map $\Phi : \R^n \to \R^n$, we introduce the set
\begin{equation} \label{eq:def-CPhi} C(\Phi) = S^{n-1} \cap \Phi(S^{n-1}) . \end{equation}

A subset $A \subset S^{n-1}$ is said to be a \emph{subsphere} if it satisfies the relation $A = S^{n-1} \cap \aff(A)$. We say that a subset of $\R^n$ is an \emph{ellipsoid} if it is a linear image of $B_n$. The following observation is fundamental to our proof.
In the three-dimensional case, it appears in \cite[Proposition IV.6]{BGNPZ14}.

\begin{lemma} \label{lemma:subsphere}
Let $\mathcal{E}$ be an ellipsoid such that $\mathcal{E} \subset B_n$. Then $\mathcal{E} \cap S^{n-1}$ is a subsphere.
\end{lemma}

\begin{proof}
Assume first that $\mathcal{E}$ is origin-symmetric. In this case, there is an orthonormal basis $(x_1,\dots,x_n)$ and numbers $\lambda_1,\dots,\lambda_n$ in $[0,1]$ such that
\[ \mathcal{E} = \left\{ \sum_{i=1}^n \lambda_i t_i x_i \st (t_1,\dots,t_n) \in B_n \right\}  .\]
It is simple to check that $\mathcal{E} \cap S^{n-1}$ equals $F \cap S^{n-1}$, where $F \subset \R^n$ is the linear subspace spanned by $\{ x_i \st \lambda_i= 1 \}$. 
This proves the lemma under the extra hypothesis that $\mathcal{E}$ is origin-symmetric.

Assume now that $\mathcal{E}$ is a general ellipsoid. If $\card (\mathcal{E} \cap S^{n-1}) \leq 1$, then $\mathcal{E} \cap S^{n-1}$ is a subsphere. Otherwise, $\mathcal{E} \cap S^{n-1}$ contains two distinct elements $x$ and $x'$.
Since the group $PO(1,n)$ of projective automorphisms of $B_n$ acts transitively on the set of lines intersecting $\mathrm{int}(B_n)$ \cite[Theorem~3.1.6]{Ratcliffe19}, we may find a projective transformation $\Theta : B_n \to B_n$ sending $x$ and $x'$ to a pair of antipodal points. The ellipsoid $\mathcal{F} = \Theta(\mathcal{E})$ intersects $S^{n-1}$ in two antipodal points and is therefore origin-symmetric. Since $\Theta$ preserves subspheres and $\mathcal{E} \cap S^{n-1} = \Theta^{-1}(\mathcal{F} \cap S^{n-1})$, we conclude by reducing to the origin-symmetric case.
\end{proof}

We now show that a primitive affine map $\Phi : \R^n \to \R^n$ satisfies $\gamma_{\aff}(B_n,\Phi) \leq n+1$. If~$\Phi$ is constant equal to $x \in B_n$, then necessarily $x \in \mathrm{int}(B_n)$ (otherwise $\Phi$ would not be primitive) and therefore $\gamma_{\aff}(B_n,\Phi) =1$. We now assume that $\Phi$ is nonconstant.

Given an integer $k \geq 0$, we set $A_k = C(\Phi^k)$. Since $\Phi$ is nonconstant, it follows from Lemma~\ref{nonconstant} that $A_{k+1} \subset A_k$. Assume that $A_{k+1}=A_k$ for some $k \geq 0$. Consider an element $x \in A_{k+1}$. There exists $y \in S^{n-1}$ such that $x=\Phi^{k+1}(y)$. The point $\Phi^k(y)$ belongs to~$A_k$, hence to $A_{k+1}$, and therefore we have $\Phi^k(y)=\Phi^{k+1}(z)$ for some $z \in S^{n-1}$. It follows that $x=\Phi^{k+2}(z)$ and thus that $x$ belongs to $A_{k+2}$. We proved that $A_{k+2}=A_{k+1}=A_k$ and therefore, by induction, $A_l=A_k$ for every $l \geq k$.
Since $\Phi$ is primitive, it follows that $A_l = \emptyset$ for every $l \geq k$.

Let $N = \gamma_{\aff}(B_n, \Phi)$ be the affine primitivity index of $\Phi$. The previous paragraph shows that
\[ \emptyset = A_N \subsetneq A_{N-1} \subsetneq \dots \subsetneq A_2 \subsetneq A_1 \subsetneq A_0 = S^{n-1}. \]
By Lemma \ref{lemma:subsphere}, each set $A_k$ is a subsphere. If two subspheres $A$, $A'$ satisfy $A \subsetneq A'$, then we have $\aff (A) \subsetneq \aff(A')$ and therefore $\dim \aff(A) < \dim \aff (A')$. The chain of inequalities
\[ 0 \leq \dim \aff(A_{N-1}) < \dots < \dim \aff(A_2) < \dim \aff(A_1) < \dim \aff(A_0) = n \]
implies that $N \leq n+1$. 

\section{A map with large maximal exponent} \label{section:lowerbound}

Our goal is to give an example of an affine map $\Phi : \R^n \to \R^n$ which is primitive and such that $\Phi^n$ is not strictly positive. Such a map satisfies $\gamma_{\mathrm{aff}}(B_n,\Phi) \geq n+1$ and, together with the result from Section \ref{section:upperbound}, allows us to conclude that $\gamma_{\mathrm{aff}}(B_n) = n+1$.

Given an angle $\theta \in [-\pi/2,\pi/2]$, we denote by $E_{n,\theta}$ the ``circle of latitude $\theta$'' defined as
\[ E_{n,\theta} = \{ (x_1,\dots,x_n) \in S^{n-1} \st x_n = \sin \theta\} .\]
Our first lemma shows that affine positive maps may send any circle of positive latitude to any circle of higher latitude. 

\begin{lemma} \label{lemma:parallels}
Let $0<\alpha<\beta<\pi/2$ and set $\lambda = \frac{\cos \beta}{\cos \alpha}$, $\mu=\frac{\tan \alpha}{\tan \beta}$. Define a map $\Psi: \R^n \to \R^n$ by the formula 
\[ \Psi : (x_1,\dots,x_n) \mapsto \left( \lambda x_1,\dots,\lambda x_{n-1}, \lambda \mu x_n + \sqrt{(1-\lambda^2)(1-\mu^2)} \right) .\]
\begin{enumerate}
    \item[(a)] The map $\Psi$ is a positive affine bijection.
    \item[(b)] If $x \in E_{n, \alpha}$, then $\Psi(x) \in E_{n, \beta}$.
    \item[(c)] If $x,y \in E_{n, \alpha}$, then $\|\Psi(x)-\Psi(y) \| = \lambda \|x-y \|$.
    \item[(d)] If $x \in B_n$ is such that $\Psi(x) \in S^{n-1}$, then $x \in E_{n, \alpha}$.
\end{enumerate}
\end{lemma}

\begin{proof}
It is immediate to check that $\Psi$ is affine and bijective, as well as property (c). Property (b) follows from the formula
$\sin \beta = \lambda \mu \sin \alpha + \sqrt{(1-\lambda^2)(1-\mu^2)}$. To check positivity of $\Psi$, it suffices to show that $\|\Psi(x)\| \leq 1$ for any $x \in S^{n-1}$. Let $\theta \in [-\pi/2,\pi/2]$ be the latitude of $x$, i.e., such that $x \in E_{n,\theta}$. We compute
\begin{eqnarray*}
1 - \| \Psi(x) \|^2 & =& 1 - \lambda^2 \cos^2 \theta - \left(\lambda \mu \sin \theta + \sqrt{(1-\lambda^2)(1-\mu^2)}\right)^2  \\
 & = & \left( \lambda \sqrt{1- \mu^2} \sin \theta - \mu \sqrt{1-\lambda^2} \right)^2 \\
 & = & \lambda^2 (1-\mu^2) (\sin \theta - \sin \alpha)^2 .
\end{eqnarray*}
The positivity of $\Psi$, together with property (d), follow from this formula.
\end{proof}

\begin{lemma} \label{lemma:gram-matrix1} 
Let $A=(a_{ij})$ be a $n \times n$ positive definite symmetric matrix satisfying $a_{ii}=1$ for every~$i$ in $\{1,\dots,n\}$. There is a number $\alpha \in (0,\pi/2)$ and vectors $x_1,\dots,x_n \in E_{n,\alpha}$ such that, for every $i,j$ in $\{1,\dots,n\}$
\[ a_{ij} = \scalar{x_i}{x_j}.\]
\end{lemma}

\begin{proof}
It is well-known \cite[Corollary 7.2.11]{hornjohnson} that we can find $y_1,\dots,y_n \in S^{n-1}$ such that $a_{ij} = \scalar{y_i}{y_j}.$ Since $A$ is invertible, the vectors $y_1,\dots,y_n$ are linearly independent and thus the hyperplane $H=\aff \{y_1,\dots,y_n\}$ does not contain $0$. We may therefore find an orthogonal transformation $Q \in O(n)$ 
such that 
\[ Q(H) = \{ (z_1,\dots,z_n) \in \R^n \st z_n = \sin \alpha \} \]
for some $\alpha \in (0,\pi/2)$. The points $x_i = Q(y_i)$ have the desired property.
\end{proof}

\begin{lemma} \label{lemma:gram-matrix2}
Consider points $x_1,\dots,x_k$ and $y_1,\dots,y_k$ in $S^{n-1}$. The following are equivalent.
\begin{enumerate}
\item There is $R \in O(n)$ such that $R(x_i)=y_i$ for every $i \in \{1,\dots,k\}$.\label{orthogonal}
\item For every $i,j$ in $\{1,\dots,k\}$, we have $\|x_i-x_j\| = \|y_i-y_j\|$. \label{isometry}
\end{enumerate} 
\end{lemma}

\begin{proof}

It is clear that (\ref{orthogonal}) implies (\ref{isometry}). Now assume that (\ref{isometry}) holds. Since all vectors involved are unit, we have $\scalar{x_i}{x_j} = \scalar{y_i}{y_j}$ for every $i,j$. Moreover, for every 
$\lambda_1,\dots,\lambda_k$ we have 
\begin{equation*}
\left\| \sum_{i=1}^k \lambda_i y_i \right\|^2 
= \sum_{i=1}^k \sum_{j=1}^k \lambda_i\lambda_j \scalar{y_i}{y_j} 
= \sum_{i=1}^k \sum_{j=1}^k \lambda_i\lambda_j \scalar{x_i}{x_j}
= \left\| \sum_{i=1}^k \lambda_i x_i \right\|^2 .
\end{equation*} This shows that the map $\hat{R}: \mathspan \{x_1,\dots,x_k\} \rightarrow \mathspan \{y_1,\dots,y_k\}$ defined by the formula
\[\hat{R} \left( \sum_{i=1}^k \lambda_i x_i \right) = \sum_{i=1}^k \lambda_i y_i
.\]
is well-defined and isometric. Finally, we extend $\hat{R}$ to a linear isometry $R: \R^n \rightarrow \R^n$ by choosing any isometry from $\mathspan \{x_1,\dots,x_k\}^\perp$ to $\mathspan \{y_1,\dots,y_k\}^\perp$. By construction, we have $R \in O(n)$ and $R(x_i) = y_i$.
\end{proof}

\begin{figure}[!htbp]
	\centering
	\includegraphics[width=3in]{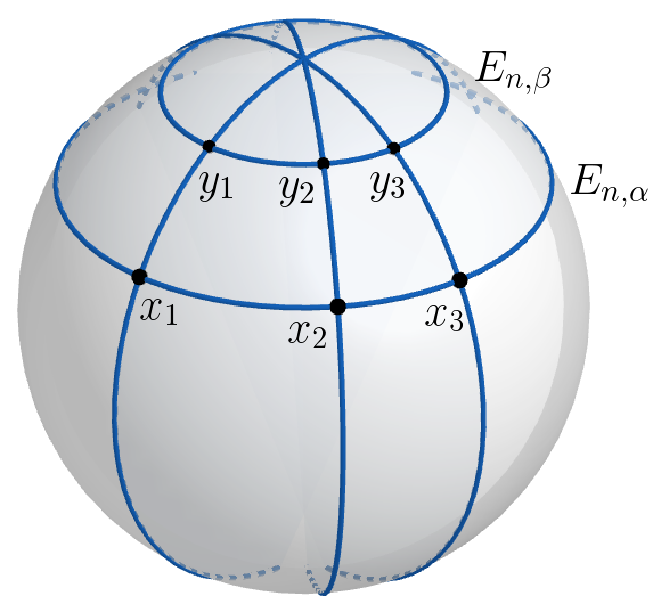}
	\caption{The affine map $\Phi$ is obtained as $R \circ \Psi$. The map $\Psi$ is a positive affine map which preserves longitude and sends a point $x_i$ with latitude $\alpha$ to a point $y_i$ with latitude $\beta>\alpha$. The map $R$ is a rotation chosen such that $R(y_1)=x_2$ and $R(y_2)=x_3$. It requires $4$ iterations of $\Phi$ from the initial point $x_1$ before reaching the interior of the unit ball.}
\end{figure}

We now construct a primitive map $\Phi$ such that $\Phi^n$ is not strictly positive. When $n=3$, an example of such a construction is depicted in Figure 1. 

Consider the following $n \times n$ matrix $A=(a_{ij})$, indexed by a parameter $c \in (0,1)$
\[ a_{ij} = \begin{cases} 1 & \textnormal{ if } i = j \\
1 - c^{\min(i,j)} & \textnormal{ if } i \neq j
\end{cases}
\]
When $c$ approaches $1$, the matrix $A$ converges to the identity matrix. We may therefore choose a value $c \in (0,1)$ such that the matrix $A$ is positive definite. By Lemma \ref{lemma:gram-matrix1}, we may find $\alpha \in (0,\pi/2)$ and vectors $x_1,\dots,x_n$ in $E_{n,\alpha}$ such that $a_{ij} = \scalar{x_i}{x_j}$. For $i \neq j$, we have
\[ \|x_i-x_j\|^2 = 2-2a_{ij} = 2 c^{\min(i,j)}. \] 
Define $\beta \in (\alpha,\pi/2)$ by the relation $\frac{\cos^2 \beta}{\cos^2 \alpha}=c$ and let $\Psi$
be the affine map given by Lemma~\ref{lemma:parallels} (applied with the present values of $\alpha$ and $\beta$). For $1 \leq i \leq n$, set $y_i = \Psi(x_i)$. By Lemma~\ref{lemma:parallels}(b), we have $y_i \in E_{n,\beta}$. For $1 \leq i < j \leq n-1$, we compute using Lemma \ref{lemma:parallels}(c)
\[ \|y_i-y_j\|^2 = \frac{\cos^2 \beta}{\cos^2 \alpha} \|x_i - x_j\|^2 = 2c^{\min(i,j)+1} = \|x_{i+1}-x_{j+1}\|^2. \]

By Lemma \ref{lemma:gram-matrix2}, there exists $R \in O(n)$ such that $R(y_i)=x_{i+1}$ for $1 \leq i \leq n-1$. We define an affine bijection $\Phi : \R^n \to \R^n$ by the formula $\Phi = R \circ \Psi$. We also set $x_{n+1}=\Phi(x_n)$, so that the relation $x_{i+1} = \Phi(x_i)$ holds for $1 \leq i \leq n$. Since $x_{n+1}=\Phi^n(x_1)$ belongs to $S^{n-1}$, it follows that $\Phi^n$ is not strictly positive.

\begin{lemma} \label{lemma:yn+1}
The point $x_0$ defined as $x_0 = \Phi^{-1}(x_1)$ does not belong to $B_n$.
\end{lemma}

\begin{proof}
Set $y_0=\Psi(x_0)=R^{-1}(x_1)$. Consider the affine hyperplanes 
\begin{align*}
V_1 = \aff \{y_0,\dots,y_{n-1} \} \\
V_2 = \aff \{x_1,\dots,x_n\} \\
V_3 = \aff \{y_1,\dots,y_{n} \}
\end{align*}
Since $V_2 \cap S^{n-1} = E_{n,\alpha}$ and $V_3 \cap S^{n-1} = E_{n,\beta}$ with $\alpha < \beta$, 
no element $S \in O(n)$ can satisfy the relation $S(V_2)=V_3$. Since $R(V_1)=V_2$, this implies that $V_1 \neq V_3$ and thus $y_0 \not\in V_3$. It follows that $y_0 \in S^{n-1} \setminus E_{n,\beta}$ and therefore that $x_0 \not\in B_n$ by Lemma \ref{lemma:parallels}(d).
\end{proof}

We now show that the map $\Phi$ is primitive by proving that $\Phi^{n+1}$ is strictly positive. As in the proof of the previous section, we denote 
\[ A_k = C(\Phi^k) = S^{n-1} \cap \Phi^k(S^{n-1}) .\]
For $1 \leq k \leq n+1$, the point $x_k$ belongs to $A_{k-1}$ (since $\Phi^{-(k-1)}(x_k)=x_1 \in S^{n-1}$) but not to $A_k$ (since $\Phi^{-k}(x_k)=x_0 \not\in S^{n-1}$ by Lemma \ref{lemma:yn+1}). This shows that $A_{k-1} \neq A_k$. 
We have therefore a chain of strict inclusions
\[ A_{n+1} \subsetneq A_n  \subsetneq \dots \subsetneq A_1 \subsetneq A_0 = S^{n-1}  \]
and therefore as in the previous section (with the convention $\dim \emptyset = -1$)
\[ \dim \aff A_{n+1} < \dim \aff A_n < \dots < \dim \aff A_1 < \dim \aff A_0 = n. \]
This is only possible if $A_{n+1}=\emptyset$. It follows that $\Phi^{n+1}$ is strictly positive.

\section{Maximal exponents for qubit channels} \label{section:qubits}

We refer to \cite{ABMB} for terminology from quantum information theory used in this section.
Given an integer $n \geq 2$, let $\M_n$ be the algebra of $n \times n$ matrices with complex entries and $\M_n^+ \subset \M_n$ be the cone positive semidefinite matrices. The maximal exponent $\gamma(\M_n^+)$ involves a supremum over positive maps (or, more precisely, over $\M_n^+$-primitive maps). However in quantum information theory it is more natural to restrict the supremum to completely positive maps and to study the quantity
\begin{equation} \label{eq:gammaCP}
\gamma^{\mathrm{cp}}(\M_n^+) := \sup \{ \gamma(\M_n^+,\Phi) \st
\Phi : \M_n \to \M_n \textnormal{ completely positive and } \M_n^+\textnormal{-primitive}
\}.
\end{equation}
This quantity appears in \cite{SPWC10,Rahaman20,MS19} in the context of the \emph{quantum Wielandt inequality}. By Proposition \ref{prop:sup-over-soles}, since the cone $\M_n^+$ is homogeneous (i.e., all its soles are affinely isomorphic to the set of quantum states), one may restrict the supremum in~\eqref{eq:gammaCP} to quantum channels, i.e., to maps which are completely positive and trace-preserving.

One obviously has $\gamma^{\mathrm{cp}}(\M_n^+) \leq \gamma(\M_n^+)$. By restricting to diagonal matrices, one has $\gamma^{\mathrm{cp}}(\M_n^+) \geq \gamma(\R_+^n)=(n-1)^2+1$. The best known upper bound $\gamma^{\mathrm{cp}}(\M_n^+) \leq C n^2 \log n$ for some constant $C$ has been proved in \cite{MS19}. To our knowledge, the quantity $\gamma(\M_n^+)$ has not been studied in the literature.
As a byproduct of our study, we obtain the following result.

\begin{theorem}
We have $\gamma(\M_2^+) = 4$ and $\gamma^{\mathrm{cp}}(\M_2^+) = 3$.
\end{theorem}

\begin{proof}
Since the cones $\M_2^+$ and $\mathcal{L}_4$ are isomorphic, the fact that $\gamma(\M_2^+)=4$ is an immediate consequence of Theorem \ref{theorem:main}.
We now explain the inequality $\gamma^{\mathrm{cp}}(\M_2^+) \leq 3$. Let $\Phi : \M_2 \to \M_2$ be a quantum channel which is $\M_2^+$-primitive. As in \eqref{eq:def-CPhi}, let $C(\Phi)$ be the set of pure states whose image under $\Phi$ is pure. A result known as the no-pancake theorem asserts that $C(\Phi)$ cannot be a circle inside the Bloch ball (see \cite[Theorem IV.9]{BGNPZ14} for a precise statement), and therefore contains at most two points. Repeating the argument from Section \ref{section:upperbound} with this extra information gives the bound $\gamma(\M_2^+,\Phi) \leq 3$.

Finally, we construct a quantum channel $\Phi$ such that $\gamma(\M_2^+,\Phi)=3$ by adapting the arguments from Section \ref{section:lowerbound}.
Given $\alpha$ and $\beta$ in $(0,\pi/2)$ such that $\alpha \neq \beta$, consider the matrices 
\[ A = \begin{pmatrix} \cos \alpha & 0 \\ 0 & \cos \beta \end{pmatrix}, \ \
B = \begin{pmatrix} 0 & \sin \beta \\ \sin \alpha & 0 \end{pmatrix},
\]
and the quantum channel $\Psi : \M_2 \to \M_2$ defined by 
$\Psi(X) = AXA^* + BXB^*$. 

Define $\theta \in (0,\pi/2)$ by the relation 
$\tan \theta = \sqrt{\sin 2\alpha/\sin 2\beta}$ and consider the vectors $\psi_+$ and $\psi_-$ in $\C^2$ defined as $\psi_{\pm} = ( \cos \theta, \pm \sin \theta)$. We claim that the states $\rho_+$ and $\rho_-$ defined as $\rho_\pm = \ketbra{\psi_\pm}{\psi_\pm}$ are the only states whose image under $\Psi$ is pure. Indeed, given a unit vector $\psi \in \C^2$, the state $\Psi(\ketbra{\psi}{\psi})$ is pure if and only if the vectors $A\ket{\psi}$ and $B\ket{\psi}$ are proportional. Our claim then follows from elementary computations.

The corresponding output states are $\Psi(\rho_\pm) = \ketbra{\phi_\pm}{\phi_\pm}$, where $\phi_{\pm} = (\cos \delta,\pm \sin \delta)$ and $\delta \in (0,\pi/2)$ is defined by the relation $\tan \delta = \sqrt{\tan \alpha/\tan \beta}$. Since $\alpha \neq \beta$, we have $\delta \neq \pi/4$ and therefore $0 < |\scalar{\phi_+}{\phi_-}| < 1$. We now use an elementary lemma. 

\begin{lemma} \label{lemma:generic-unitary}
Let $\phi_+,\phi_-,\psi_+,\psi_-$ be unit vectors in $\C^2$ such that $0 < |\scalar{\phi_+}{\phi_-}| < 1$. Then there exists a unitary matrix $U$ such that $U(\phi_+)=\psi_-$ and $U(\phi_-)$ is neither proportional to $\psi_+$ nor to $\psi_-$.
\end{lemma}

\begin{proof}
Write $\phi_- = a \phi_+ +b \chi$ where $\chi$ is a unit vector orthogonal to $\phi_+$ and $a,b$ are complex numbers such that $|a|^2+|b|^2=1$. Pick a unit vector $\omega$ orthogonal to $\psi_-$. Since $a$ and $b$ are nonzero, we may choose $\theta \in \R$ such that $a \psi_- + be^{i\theta} \omega$ is neither proportional to $\psi_+$ nor to~$\psi_-$. The unitary matrix sending the basis $(\phi_+,\chi)$ to the basis $(\psi_-,e^{i\theta}\omega)$ has the desired property.
\end{proof}

Let $U$ be a unitary matrix given by the lemma and consider the quantum channel $\Phi$ defined by $\Phi(X)= U\Psi(X)U^*$. The only states with a pure output under $\Phi$ are $\rho_+$ and~$\rho_-$. Moreover, $\Phi(\rho_+)= U \ketbra{\phi_+}{\phi_+} U^* = \rho_-$ and $\Phi(\rho_-)= U \ketbra{\phi_-}{\phi_-} U^*$ is  a pure state which, by Lemma \ref{lemma:generic-unitary}, is distinct from $\rho_+$ and $\rho_-$. It follows that $\Phi^2(\rho_-)=\Phi^3(\rho_+)$ is not pure. Since $\Phi^3$ is strictly positive and $\Phi^2$ is not, the channel $\Phi$ has a maximal index equal to~$3$.
\end{proof}

\bibliographystyle{plain}
\bibliography{reference}{}

\end{document}